\documentclass[letterpaper, 12pt]{amsart}

\usepackage[T1]{fontenc}

\usepackage[margin=1in]{geometry}
 
\usepackage{tikz}

\usetikzlibrary{shadings}

\usepackage{amscd}

\usepackage{amsmath}

\usepackage{amsfonts}

\usepackage{amssymb}

\usepackage{amsthm}

\usepackage{bbold}

\usepackage{mathtools}

\usepackage{arydshln}

\usepackage{fancyhdr}

\usepackage{verbatim}

\usepackage[all]{xy}

\usepackage{enumitem}

\usepackage{color}

\usepackage[colorlinks=true, linkcolor=teal, citecolor=-teal,
pagebackref=true]{hyperref}

\usepackage[reftex]{theoremref}

\usepackage{nameref}

\usepackage{charter}
\usepackage{newtxmath}

\theoremstyle{plain}
\newtheorem*{theorem*}{Theorem}
\newtheorem{theorem}{Theorem}
\newtheorem{lemma}[theorem]{Lemma}
\newtheorem{corollary}[theorem]{Corollary}
\newtheorem{proposition}[theorem]{Proposition}
\newtheorem{conjecture}{Conjecture}
\newtheorem*{claim*}{Claim}

\theoremstyle{definition}

\theoremstyle{remark}
\newtheorem*{remark*}{Remark}

\newcommand{\RR}{\mathbb{R}}

\newcommand{\QQ}{\mathbb{Q}}

\newcommand{\NN}{\mathbb{N}}

\newcommand{\ZZ}{\mathbb{Z}}

\newcommand{\calQ}{\mathcal{Q}}

\newcommand{\calS}{\mathcal{S}}

\newcommand{\bone}{\mathbf{1}}

\renewcommand{\leq}{\leqslant} \renewcommand{\geq}{\geqslant}

\DeclarePairedDelimiter{\abs}{\lvert}{\rvert}
\DeclarePairedDelimiter{\norm}{\lVert}{\rVert}

\DeclarePairedDelimiter{\set}{\lbrace}{\rbrace}
\DeclarePairedDelimiter{\parens}{\lparen}{\rparen}
\DeclarePairedDelimiter{\brackets}{\lbrack}{\rbrack}

\DeclareMathOperator{\meas}{m}

\def\eps{{\varepsilon}}

\def\1int{{[0,1]}}

\title[Toward Khintchine's theorem with a moving target]{Toward
  Khintchine's theorem with a moving target: extra divergence or
  finitely centered target}

\author{Gilbert Michaud \and Felipe A.~Ram{\'i}rez}

\address{Department of Mathematics and Computer Science, Wesleyan
  University, CT}

\email{gmichaud@wesleyan.edu}

\email{framirez@wesleyan.edu}

\date{}



\begin{document}

\begin{abstract}
  Sz{\"u}sz's inhomogeneous version (1958) of Khintchine's theorem
  (1924) gives conditions on $\psi:\NN\to\RR_{\geq 0}$ under which for
  almost every real number $\alpha$ there exist infinitely many
  rationals $p/q$ such that
  \begin{equation*}
    \abs*{\alpha - \frac{p+\gamma}{q}} < \frac{\psi(q)}{q},
  \end{equation*}
  where $\gamma\in\RR$ is some fixed inhomogeneous parameter. It is
  often interpreted as a statement about visits of
  $q\alpha\,(\bmod 1)$ to a shrinking target centered around
  $\gamma\,(\bmod 1)$, viewed in $\RR/\ZZ$. Hauke and the second
  author have conjectured that Sz{\"u}sz's result continues to hold if
  the target is allowed to move as well as shrink, that is, if the
  inhomogeneous parameter $\gamma$ is allowed to depend on the
  denominator $q$ of the approximating rational. We show that the
  conjecture holds under an ``extra divergence'' assumption on
  $\psi$. We also show that it holds when the inhomogeneous
  parameter's movement is constrained to a finite set. As a byproduct,
  we obtain a finite-colorings version of the inhomogeneous Khintchine
  theorem, giving rational approximations with monochromatic denominators.
\end{abstract}

\maketitle

\section{Introduction}

For a function $\psi:\NN\to\RR_{\geq 0}$ and a real number
$\gamma\in\RR$, define
\begin{equation*}
  W(\psi, \gamma) = \set*{\alpha\in[0,1] : \norm{q\alpha - \gamma} < \psi(q) \textrm{ for infinitely many } q\in \NN},
\end{equation*}
where $\norm{\cdot}$ denotes distance to $\ZZ$. The elements of
$W(\psi,\gamma)$ are called the \emph{$\psi$-approximable numbers with
  inhomogeneous parameter $\gamma$}.

This definition has a handy dynamical interpretation where one views
$[0,1]$ as a fundamental domain of $\RR/\ZZ$. Then $\alpha$ is an
element of $W(\psi,\gamma)$ if the rotation orbit
$\set{q\alpha\,(\bmod 1)}_{q=1}^\infty$ enters the shrinking target
$(\gamma-\psi(q), \gamma + \psi(q))\, (\bmod\, 1)$ infinitely many times.

In metric Diophantine approximation we are interested in the Lebesgue
measure of $W(\psi,\gamma)$ and of similarly defined sets. The
following classical theorem gives conditions under which
$W(\psi, \gamma)$ has measure $0$ or measure $1$.

\begin{theorem*}[Inhomogeneous Khintchine Theorem,~\cite{SzuszinhomKT}]
  If $\psi:\NN\to\RR_{\geq 0}$ is decreasing and $\gamma\in\RR$, then
  \begin{equation*}
    \meas(W(\psi, \gamma)) =
    \begin{cases}
      0 &\textrm{if } \sum\psi(q) < \infty \\
      1 &\textrm{if } \sum\psi(q) =\infty,
    \end{cases}
  \end{equation*}
  where $\meas$ denotes Lebesgue measure. 
\end{theorem*}

\begin{remark*}
  The convergence part is a standard application of the
  Borel--Cantelli lemma. Throughout this paper, we are only concerned
  with the divergence part.
\end{remark*}

The inhomogeneous Khintchine theorem has been generalized to higher
dimensions~\cite{Schmidt} and to systems of linear
forms\cite{Sprindzuk}. Recent
papers~\cite{allen2021independence,allen2023inhomogeneous,hauke2024duffinschaefferconjecturemovingtarget}
have focused special attention on the fact that these higher
dimensional generalizations allow for the inhomogeneous parameter
$\gamma$ to depend on the denominator $q$---the target can move. Yet,
it is still not known whether the one dimensional theorem quoted above
allows for a moving target. The following problem is posed
in~\cite[Conjecture~2]{hauke2024duffinschaefferconjecturemovingtarget}.

\begin{conjecture}[Khintchine's theorem with a moving target]\th\label{conj:movingKT}
  For $\gamma:=(\gamma_q)_{q}$ a sequence of real numbers and
  $\psi:\NN\to\RR_{\geq 0}$ a function, define
  \begin{equation*}
    W(\psi, \gamma) = \set*{\alpha\in[0,1] : \norm{q\alpha - \gamma_q} < \psi(q) \textrm{ infinitely often}}.
  \end{equation*}
  If $\psi:\NN\to\RR_{\geq 0}$ is decreasing and
  $\sum \psi(q) = \infty$, then $\meas(W(\psi, \gamma))=1$.
\end{conjecture}

\begin{remark*}
  Note that in the notation, $\gamma$ is now the name of the sequence
  $(\gamma_q)_{q\geq 1}$. Henceforth, when a real number
  $\sigma\in\RR$ appears in the notation $W(\psi,\sigma)$, we intend
  for it to be interpreted as the constant sequence
  $(\sigma, \sigma, \dots)$.
\end{remark*}

In this paper we establish~\th~\ref{conj:movingKT} under an extra
divergence assumption (\th~\ref{thm:extradiv}) and under the
assumption that the sequence $(\gamma_q)_q$ lies in a finite set
(\th~\ref{thm:finitetarget}).

\subsection{Extra divergence}

Sprind{\v z}uk~\cite{Sprindzuk} gives an asymptotic formula for the number
of solutions to $\norm{q\alpha-\gamma_q} < \psi(q)\, (q\leq Q)$, and
this formula shows that~\th~\ref{conj:movingKT} holds if
\begin{equation*}
\sum_{q\leq Q} \psi(q) \gg (\log Q)(\log\log Q)^{3+\delta}
\end{equation*}
(see \th~\ref{prop:sprindzuk}). But this rate is too fast to even
apply to the function $1/q$. We prove the following extra divergence
result, which does apply in particular to $1/q$.

\begin{theorem}\th\label{thm:extradiv}
  Let $\gamma:=(\gamma_q)_q$ be a sequence of real numbers and $\psi$
  decreasing. If there exists $\eps>0$ and $k\geq 2$ for which
  \begin{equation*}
    \sum_{q=1}^\infty \frac{\psi(q)}{\sqrt{\log q}(\log\log q) (\log\log\log q) \dots (\underbrace{\log\log\dots\log}_{k \textrm{ iterates}} q)^{1+\eps}} = \infty,
  \end{equation*}
  then $\meas(W(\psi,\gamma))=1$.  In particular, $\sum \psi(q)(\log q)^{-\frac{1}{2}- \eps}=\infty$ suffices.
\end{theorem}

\begin{remark*}
  Throughout this paper, we use the notation
  $\log x = \max\set{1, \ln x}$, where $\ln$ is the natural logarithm.
\end{remark*}

The following corollary is immediate.

\begin{corollary}\th\label{cor:inparticular}
  \th~\ref{conj:movingKT} holds for any monotonic $\psi$ such that
  \begin{equation*}
    \psi(q) \gg \frac{(\overbrace{\log\log\dots\log}^{k \textrm{ iterates}} q)^\eps}{q \sqrt{\log q}}.
  \end{equation*}
  In particular, it holds for $\psi(q) = 1/q$. 
\end{corollary}

\begin{remark*}[On asymptotic notation]
  The notation $f(q)\ll g(q)$ means that there is some constant $C>0$
  such that $f(q) \leq C g(q)$ for all relevant $q$, that is, if
  $f(q) = O(g(q))$. We use $f(q)\asymp g(q)$ to mean that both
  $f(q) \ll g(q)$ and $g(q) \ll f(q)$ hold. 
\end{remark*}

\th~\ref{thm:extradiv} is reminiscent of a collection of extra
divergence
results~\cite{HPVextra,BHHVextraii,AistleitneretalExtraDivergence}
that were proved in the decade before Koukoulopoulos and Maynard's
2020 proof of the Duffin--Schaeffer
conjecture~\cite{KMDS,duffinschaeffer}. By direct analogy with the
last of these~\cite{AistleitneretalExtraDivergence}, one may expect
that the extra divergence in \th~\ref{thm:extradiv} can be relaxed to
$\sum \psi(q)(\log q)^{-\eps}=\infty$. However, the problems are quite
different, as are the methods of proof. The inhomogeneity of our
setting is directly responsible for the $\sqrt{\log q}$ appearing in
our extra divergence assumption, and it is not obvious how (or
whether) one can avoid it.

For readers interested in Duffin--Schaeffer-type problems, we remark
that the Duffin--Schaeffer conjecture with a moving target (i.e., a
Duffin--Schaeffer version of~\th~\ref{conj:movingKT}) has already been
ruled
out~\cite[Theorem~2]{hauke2024duffinschaefferconjecturemovingtarget}.

\subsection{Finitely centered target and monochromatic denominators}

We are able to prove~\th~\ref{conj:movingKT} unconditionally on $\psi$
if the target centers $\gamma = (\gamma_q)$ lie in a finite set. This
follows from the following more general statement.

\begin{theorem}\th\label{thm:finitetarget}
  Let $\set{\pi_1, \pi_2, \dots, \pi_\ell}$ be a partition of $\NN$,
  and $\sigma_1, \sigma_2, \dots, \sigma_\ell\in \RR$. If
  $\psi:\NN\to\RR_{\geq 0}$ is decreasing and $\sum\psi(q)=\infty$,
  then there exists $k \in\set{1, \dots, \ell}$ such that
  $\meas(W(\bone_{\pi_k}\psi, \sigma_k))=1$.
\end{theorem}

As immediate corollary, there is the following statement
addressing~\th~\ref{conj:movingKT}.

\begin{corollary}[Khintchine's theorem with a finitely centered moving target]\label{cor:finite}
  \th~\ref{conj:movingKT} is true if the sequence $(\gamma_q)_q$ is
  contained in a finite subset
  \begin{equation*}
    (\gamma_q)_q\subset\set{\sigma_1, \sigma_2, \dots,
      \sigma_\ell}\subset \RR.
  \end{equation*}
  In this case, more is true: with $\psi$ as
  in~\th~\ref{conj:movingKT}, there exists $\sigma \in\set{\sigma_1, \dots, \sigma_\ell}$ 
  such that $\meas(W(\bone_{\set{\gamma_q = \sigma}}\psi, \gamma))=1$.
\end{corollary}

\begin{proof}
  Define $\pi_k = \set{q \in \NN : \gamma_q = \sigma_k}$ and
  apply~\th~\ref{thm:finitetarget}. 
\end{proof}

Another corollary to~\th~\ref{thm:finitetarget} is the following,
which can be expressed in terms of colorings of $\NN$. If every
$q\in\NN$ is assigned one of finitely-many possible colors, and $\psi$
is a monotonic function such that $\sum\psi(q)$ diverges, then there
is some color---say, orange---for which almost every real number is
$\psi$-approximable by rationals with orange denominators.

\begin{corollary}[Inhomogeneous Khintchine theorem with
  monochromatic denominators]\th\label{monochromatic}
  Let $\gamma\in\RR$ and $\set{\pi_1, \dots, \pi_\ell}$ a partition of
  $\NN$. Let $\psi$ be as in~\th~\ref{conj:movingKT}. Then there
  exists a partition element $\pi \in \set{\pi_1, \dots, \pi_\ell}$
  such that $\meas(W(\bone_{\pi}\psi, \gamma))=1$. 
\end{corollary}

\begin{proof}
  Apply~\th~\ref{thm:finitetarget} with $\sigma_k=\gamma$ for
  $k=1, \dots, \ell$.
\end{proof}

\begin{remark*}
  The homogeneous version of~\th\ref{monochromatic} admits a simpler
  proof. One notes that
  \begin{equation*}
    W(\psi,0) = W(\bone_{\pi_1}\psi, 0) \cup\dots\cup W(\bone_{\pi_\ell}\psi, 0).
  \end{equation*}
  By Khintchine's theorem, $W(\psi,0)$ has full measure, hence for
  some $k=1, \dots, \ell$, one of the $W(\bone_{\pi_k}\psi, 0)$ has
  positive measure. Then, by Cassels' zero-one
  law~\cite[Theorem~I]{Casselslemma}, it has full measure. But there
  is no inhomogeneous zero-one law. It is a well-known problem to prove one.
\end{remark*}

One can compare \th~\ref{monochromatic} with results in~\cite{Harman}
where Khintchine-type statements are proved under restrictions on
numerators and denominators. Among those, it is shown that if
$D\subset\NN$ is a set with positive lower asymptotic density, and
$\psi$ is a monotonic function such that $\sum\psi(q)$ diverges, then
almost every $\alpha$ is $\psi$-approximable by rationals with
denominators lying in $D$. The monochromatic set
in~\th~\ref{monochromatic} need not have positive density, but it may
depend on $\psi$. This dependence is unavoidable, as the following
simple example shows.

Let $\gamma\in\RR$ and consider the partition $\NN=\pi_0\cup\pi_1$ where
\begin{equation*}
  \pi_0 = \bigcup_{\substack{k\geq 0 \\ k \textrm{ even}}}[2^{2^k}, 2^{2^{k+1}})\cap \NN \qquad\textrm{and}\qquad
  \pi_1 = \set{1}\cup\bigcup_{\substack{k\geq 0 \\ k \textrm{ odd}}}[2^{2^k}, 2^{2^{k+1}})\cap \NN.
\end{equation*}
Let
\begin{equation*}
  \psi_0(q) =
  \begin{cases}
    2^{-2^{k+1}} &\textrm{if } q\in [2^{2^k}, 2^{2^{k+1}}), \textrm{with $k\geq 0$ even} \\
    q^{-2} &\textrm{otherwise}
  \end{cases}
\end{equation*}
and
\begin{equation*}
  \psi_1(q) =
  \begin{cases}
    2^{-2^{k+1}} &\textrm{if } q\in [2^{2^k}, 2^{2^{k+1}}), \textrm{with $k\geq 0$ odd} \\
    q^{-2} &\textrm{otherwise.}
  \end{cases}
\end{equation*}
Then both $\psi_0$ and $\psi_1$ are decreasing and have diverging
sums. Since $\sum q^{-2}$ converges, the Borel--Cantelli lemma implies
that
$\meas(W(\bone_{\pi_0}\psi_1, \gamma)) = \meas(W(\bone_{\pi_1}\psi_0, \gamma)) =
0$. Therefore, the monochromatic set coming
from~\th~\ref{monochromatic} is $\pi_0$ for $\psi_0$ and $\pi_1$ for
$\psi_1$---that is, it must depend on the approximating function.

\subsection{Other remarks}

It is worth noting that other natural variants
of~\th~\ref{conj:movingKT} are also open. For instance, it does not
seem easier to prove the conjecture in the case where the sequence
$\gamma = (\gamma_q)$ converges to some $\gamma_0\in\RR$. But for
convergent sequences, one can easily name conditions under which
$W(\psi, \gamma)$ has full measure. For example, if
$\abs{\gamma_q - \gamma_0} \ll \psi(q)$ for all $q$, then, by the
triangle inequality, $W(\psi, \gamma_0)\subset W(C\psi, \gamma)$ for
some $C>0$. The inhomogeneous Khintchine theorem gives
$\meas(W(\psi, \gamma_0))=1$, hence $\meas(W(C\psi, \gamma))=1$, and
the result follows from Cassels'
lemma~\cite[Lemma~9]{Casselslemma}. This idea can be
pushed. In~\th\ref{conj:movingKT} we can assume without loss of
generality that $\psi(q) \geq c(q)$ for every nonincreasing sequence
$c(q)$ such that $\sum c(q) < \infty$. (One simply replaces $\psi(q)$
with $\max\set{\psi(q),c(q)}$. The divergence condition persists, and
the difference in the resulting limsup sets has measure $0$ by the
convergence of $\sum c(q)$ and the Borel--Cantelli lemma.) The
preceding argument now implies that~\th~\ref{conj:movingKT} holds for
$\gamma = (\gamma_q)$ if there exists $\gamma_0\in\RR$ such that
$\sum_q \abs{\gamma_q - \gamma_0} < \infty$. We leave open the
question for sequences with a slower rate of convergence.

A similar argument shows that if there is a dense set
$\calQ\subset\RR$ such that~\th~\ref{conj:movingKT} holds for all
moving targets whose centers lie in $\calQ$, then the full conjecture
holds. For suppose $\gamma:=(\gamma_q)\subset\RR$ and $\psi$ is
decreasing with $\sum \psi(q)=\infty$. By the density of $\calQ$,
there exists a sequence $\sigma:=(\sigma_q)_q\subset\calQ$ such that
$\abs{\gamma_q - \sigma_q} \leq \psi(q)$ for all $q$. Again the
triangle inequality shows that $W(\psi, \sigma)$ is contained in
$W(2\psi, \gamma)$. The former has full measure by assumption,
therefore so does $W(2\psi, \gamma)$, and the result follows Cassels'
lemma as before. This reduces~\th~\ref{conj:movingKT} to the case
where the target centers are, say, rational.

The reduction to dense subsets of $\RR$ is so immediate that one may
take it as a hint that \th~\ref{conj:movingKT} for rational
sequences---or sequences contained in any other dense subset
$\calQ\subset\RR$---is no easier than the full conjecture. If so, we
ask the following intermediate question: Can one find a set
$\calS\subset \RR$ such that $\calS\, (\bmod 1)$ is infinite and such that \th~\ref{conj:movingKT} holds for all
sequences $\gamma:=(\gamma_q)\subset \calS$?

\section{Measure theoretic statements}

In this section we gather some general measure-theoretic results that
are indispensable in this work. The sets $W(\psi,\gamma)$ are best
understood as limsup sets, and the results listed in this section are
tools for showing that the measures of limsup sets are large.

The first is often called the Divergence Borel--Cantelli lemma, and
can be found in various sources, such as \cite[Lemma~2.3]{Harman}. It
is often attributed to Erd{\H o}s--Renyi~\cite{ErdosRenyiBC}.

\begin{proposition}[Divergence Borel--Cantelli
  lemma]\th\label{divergenceBC}
  Let $(X,\mu)$ be a finite measure space and $(A_q)_{q\geq 1}$ a
  sequence of measurable subsets of $X$ such that
  $\sum_{q\geq 1}\mu(A_q)=\infty$. Then
\begin{equation*}
    \mu(\limsup_{q\to\infty}A_q) \geq \limsup_{Q\to\infty}\frac{\parens*{\sum_{q\leq Q}\mu(A_q)}^2}{\sum_{q,r\leq Q}\mu(A_q\cap A_r)}.
\end{equation*}
\end{proposition}

If the right-hand side of the inequality appearing in~\th\ref{divergenceBC} is positive, one says that the sets $A_q$ are \emph{quasi-independent on average (QIA)}. Quasi-independence on average is the crux of many results in metric Diophantine approximation, including the main results here. Once QIA is established, one has a positive-measure limsup set. 

Closely related to the Divergence Borel--Cantelli lemma is the Erd{\H o}s--Chung lemma. 

\begin{proposition}[Chung--Erd\H{o}s Lemma,~\cite{chungerdos}]\th\label{prop:erdoschung}
  If $(X, \mu)$ is a finite measure space and
  $(A_q)_{q \in \NN}\subset X$ is a sequence of measurable subsets
  such that $0 < \sum_{q} \mu(A_q) < \infty$, then
  \begin{equation*}
    \mu\parens*{\bigcup_{q=1}^\infty  A_q} \geq \frac{\parens*{\sum_{q=1}^\infty \mu(A_q)}^2}{\sum_{q,r=1}^\infty \mu(A_q\cap A_r)}.
  \end{equation*}
\end{proposition}

\begin{remark*}
    This is usually stated with finitely many sets $A_1, \dots, A_Q$. The version stated above is a simple extension. (See~\cite[Proposition~3]{ramirez2023duffinschaeffer}.)
\end{remark*}

The challenge in applying \th\ref{divergenceBC,prop:erdoschung} to problems in Diophantine approximation comes from the fact that the sets arising in these problems are far from independent in general. Commonly, rather than having a bound of the form $\mu(A_q\cap A_r) \ll \mu(A_q)\mu(A_r)$, which would be useful for the application of the above results, one instead tends to find an extra term, 
\begin{equation}\label{eq:extraterm}
    \mu(A_q\cap A_r) \ll \mu(A_q)\mu(A_r) + \textrm{extra},
\end{equation}
and the problem becomes to control that term. In homogeneous problems,
the extra term comes from rational numbers expressible with both
denominator $q$ and denominator $r$, and it can be dealt with by
requiring that all rational numbers be expressed in reduced
form. (This then leads to other difficulties. It is the reason that a
monotonicity assumption appears in Khintchine's theorem, and that
monotonicity assumption in turn gave rise to the Duffin--Schaeffer
conjecture~\cite{duffinschaeffer}, a problem with a rich 80-year
history~\cite{duffinschaeffer,AistleitnerDS,HPVextra,PollingtonVaughan,BHHVextraii,BVmassTP,chow2023littlewood}
culminating in a proof by Koukoulopoulos and Maynard~\cite{KMDS} and
inspiring follow-up
work~\cite{ramirez2023duffinschaeffer,hauke2024duffinschaefferconjecturemovingtarget}.)

For inhomogeneous problems, simply reducing fractions does not remove
the extra term in~\eqref{eq:extraterm}. Nevertheless, if the
inhomogeneous parameter is fixed, then one can benefit by a shifted
kind of reduction, as we will see in the proof
of~\th\ref{thm:finitetarget}, which uses a reduction of
Schmidt~\cite{Schmidt}. In the moving target setting, where the
inhomogeneous parameter is allowed to depend on the denominator, it
becomes even harder to handle the extra term. The following result
shows that it can be absorbed if the sum of measures of $A_q$ exhibits
extra divergence. This is an abstraction of Yu's proof
of~\cite[Theorem~1.8]{Yu}.

\begin{proposition}\th\label{prop:abstractyu}
  For each $q\geq 1$, let $A_q\subset X$, where 
  $X$ is a metric space with a finite measure $\mu$ such that every open set is $\mu$-measurable.
  Suppose $\eta:\NN^2\to[1,\infty)$ is such that
  \begin{equation}\label{eq:generaloverlaps}
    \mu(A_q\cap A_r) \leq\kappa \mu(A_q)(\mu(A_r) + \eta(q,r))\qquad (1 \leq r\leq q)
  \end{equation}
  for some $\kappa>0$.  Suppose $U\subset X$ is an open set and that
  \begin{equation*}
      \mu(A_q\cap U) \geq \frac{1}{2}\mu(A_q)\mu(U) 
  \end{equation*}
  for all sufficiently large $q$. Let $\eta(q) = \sum_{r\leq q} \eta(q,r)$ and
  $h:[1,\infty)\to [1,\infty)$ increasing and such that
  \begin{equation}
    \label{eq:hconv}
    h(2q) \ll h(q) \quad\textrm{and}\quad \sum_{\ell \geq 0} \frac{1}{h(2^\ell)} < \infty.
  \end{equation}
 If
  \begin{equation}\label{eq:etasum}
    \sum_{q=1}^\infty \frac{\mu(A_q)}{\eta(q)h(\eta(q))} = \infty
  \end{equation}
  then $\mu(\limsup A_q\cap U)\geq \mu(U)^2/(4\kappa)$.
\end{proposition}

\begin{proof}
  Following Yu's argument
  in~\cite{Yu}, we define for each $\ell\geq 0$ the set
  \begin{equation*}
    D_\ell = \set{q\in\NN : 2^\ell \leq \eta(q) < 2^{\ell+1}}. 
  \end{equation*}
  If there exists $\ell\geq 0$ for which
  $\sum_{q\in D_\ell}\mu(A_q)=\infty$, then we may restrict attention
  to $D_\ell$.  By~(\ref{eq:generaloverlaps}),
  \begin{align*}\label{eq:generalavgoverlaps}
    \sum_{\substack{q,r=1 \\ q,r\in D_\ell}}^Q\mu(A_q\cap A_r\cap U)
    &\leq \sum_{\substack{q,r=1 \\ q,r\in D_\ell}}^Q\mu(A_q\cap A_r) \\
    &\leq 2\sum_{\substack{r < q \\ q,r \in D_\ell}}\parens*{\kappa \mu(A_q)(\mu(A_r) + \eta(q,r)} + 2\kappa \sum_{q\in D_\ell} \mu(A_q)  \\
    &\leq \kappa \parens*{\sum_{\substack{q=1 \\ q\in D_\ell}}^Q\mu(A_q)}^2(1+o(1)) + 2\kappa \sum_{\substack{q=1 \\ q\in D_\ell}}^Q\mu(A_q)\eta(q) \\
    &\leq \kappa \parens*{\sum_{\substack{q=1 \\ q\in D_\ell}}^Q\mu(A_q)}^2(1 + o(1)) + 2 \kappa \sum_{\substack{q=1 \\ q\in D_\ell}}^Q\mu(A_q)2^{\ell + 1} \\
    &= \kappa \parens*{\sum_{\substack{q=1 \\ q\in D_\ell}}^Q\mu(A_q)}^2 (1 + o(1)) \\
    &= \frac{4\kappa}{\mu(U)^2} \parens*{\sum_{\substack{q=1 \\ q\in D_\ell}}^Q\mu(A_q\cap U)}^2 (1 + o(1))
  \end{align*}
  and by~\th~\ref{divergenceBC}, it follows that
  $\mu(\limsup A_q)\geq \mu(U)^2/(4\kappa)$.

  Suppose there is no such $\ell$.  For each $\ell\geq 0$, denote by
  $\Sigma_\ell$ the sum~(\ref{eq:etasum}) extended over $q\in
  D_\ell$. Then, by assumption, $\Sigma_\ell < \infty$ for each
  $\ell\geq 0$ and $\sum_{\ell\geq 0}\Sigma_\ell = \infty$.  By~\th\ref{prop:erdoschung},
  \begin{align}
    \mu\parens*{\bigcup_{q \in D_\ell} A_q \cap U}
    &\geq \frac{\parens*{\sum_{q \in D_\ell} \mu(A_q\cap U)}^2}{\sum_{q,r \in D_\ell} \mu(A_q\cap A_r\cap U)} \\
    &\geq \frac{\mu(U)^2}{4\kappa}\cdot \frac{\parens*{\sum_{q \in D_\ell} \mu(A_q)}^2}{\parens*{\sum_{q\in D_\ell}\mu(A_q)}^2(1 + o(1)) + 2 \sum_{q\in D_\ell}\mu(A_q)\eta(q)}. \nonumber
  \end{align}
    Now,
    \begin{equation*}
      \sum_{q\in D_\ell}\mu(A_q)\eta(q)
      \leq 2^{2(\ell+1)}h(2^{\ell+1})\Sigma_\ell\ll 2^{2(\ell+1)}h(2^\ell)\Sigma_\ell
    \end{equation*}
    and
    \begin{equation*}
      \parens*{\sum_{q\in D_\ell}\mu(A_q)}^2 \geq 2^{2\ell}h(2^\ell)^2\Sigma_\ell^2.
    \end{equation*}
  By the convergence assumed in~(\ref{eq:hconv}), it follows that
  $h(2^\ell) \Sigma_\ell$ gets arbitrarily large, and this implies that
  \begin{equation*}
    2 \sum_{q\in D_\ell}\mu(A_q)\eta(q) = o\brackets*{\parens*{\sum_{q\in D_\ell}\mu(A_q)}^2},
  \end{equation*}
  hence
    \begin{align*}
    \mu\parens*{\bigcup_{q \in D_\ell} A_q}
    &\geq\frac{\mu(U)^2}{4\kappa} \frac{\parens*{\sum_{q \in D_\ell} \mu(A_q)}^2}{ \parens*{\sum_{q\in D_\ell}\mu(A_q)}^2(1 + o(1))}
    \end{align*}
  for infinitely many $\ell$. Once again, it follows 
  that $\mu(\limsup A_q)\geq \mu(U)^2/(4\kappa)$.
\end{proof}

In the proofs of \th\ref{thm:extradiv,thm:finitetarget}, once it is established that $W(\psi,\gamma)$ has positive measure, it remains to show that its measure is $1$. Once again, this is a task that is easier for homogeneous problems than it is for inhomogeneous problems, because there are zero-one laws in the homogeneous setting~\cite{Casselslemma,Gallagher01,Vilchinski}. For inhomogeneous problems, it is usually necessary to establish positive measure on small scales and use a variant of the Lebesgue density lemma. We use the following result for~\th\ref{thm:extradiv}. 
 
\begin{proposition}[Beresnevich--Dickinson--Velani,~{\cite[Lemma 6]{BDV}}]\th\label{BDVdensitylemma}
  Let $(X,d)$ be a metric space with a finite measure $\mu$ such that
  every open set is $\mu$-measurable. Let $A$ be a Borel subset of $X$
  and let $f:\RR_+\to\RR_+$ be an increasing function with $f(x)\to 0$ as
  $x\to 0$. If for every open set $U\subset X$ we have
  \begin{equation*}
    \mu(A\cap U) \geq f(\mu(U)),
  \end{equation*}
  then $\mu(A) = \mu(X)$.
\end{proposition}

We use the following result in the proof of~\th\ref{thm:finitetarget}.
\begin{proposition}[Beresnevich--Hauke--Velani,~{\cite[Theorem 5]{beresnevich2024borel}}]\th\label{BHV}
Let $\mu$ be a doubling Borel regular probability measure on a metric space $X$. Let $(A_q)_{q\in\NN}$ be a sequence of $\mu$-measurable subsets of $X$. Suppose that 
\begin{equation}\label{eqn01}
\sum_{q=1}^\infty \mu(A_q)=\infty
\end{equation}
and that there exists a constant $C>0$ such that
\begin{equation}\label{eqn02}
\sum_{q,r=1}^Q  \mu(A_q\cap A_r)\leq C\left(\sum_{q=1}^Q  \mu(A_q)\right)^2\quad\text{for infinitely many $Q\in\NN$\,.}
\end{equation}
 In addition, suppose that for any $\delta>0$ and any closed ball $B$ centered at $\operatorname{supp}\mu$ there exists $q_0 =q_0 (\delta, B) $ such that for all $q  \geq q_0$
\begin{equation}\label{vb89}
\mu\left(B\cap A_q\right)\le (1+\delta)\mu\left(B\right)\mu(A_q)\,.
\end{equation}
Then $\mu(\limsup A_q)=1$.
\end{proposition}

\section{Number theoretic statements}

The main purpose of this section is to prove the following statement
relating the weights that will arise in the application
of~\th~\ref{prop:abstractyu} to the weights appearing
in~\th~\ref{thm:extradiv}. We also collect some standard number
theoretic lemmas.

\begin{lemma}\th\label{loglogsum}
  For $d(n)$ the number of divisors of $n$ and
  $f:[1,\infty)\to[1,\infty)$ any increasing function, one has
  \begin{equation*}
    \sum_{n\leq x}\frac{1}{f(\log d(n))d(n)} 
    \gg
    \frac{x}{f((\log 2)\log\log x)\sqrt{\log x} }. 
  \end{equation*}
\end{lemma}

\begin{proof}
  By a result of Kac~\cite{Kacdivisorfundist}, for all $y\in\RR$, we have
  \begin{equation*}
    \lim_{x\to\infty} \frac{1}{x}\#\set*{1\leq n \leq x : \log_2 d(n) \leq \log\log x + y \sqrt{\log\log x}} = \frac{1}{2\pi}\int_{-\infty}^y e^{-w^2/2}\, dw.
  \end{equation*}
  In particular, by setting $y=0$ we see that
  \begin{equation*}
    \#\underbrace{\set*{1\leq n \leq x : \log_2 d(n) \leq \log\log x}}_{A(x)} \sim \frac{x}{2}\quad (x\to\infty),
  \end{equation*}
  and it follows that for all $x$ large enough we have
  \begin{equation*}
    \sum_{n \in A(x)} \frac{1}{d(n)} \geq \frac{1}{3} \sum_{n \leq x} \frac{1}{d(n)}. 
  \end{equation*}
  Observe then that
  \begin{align*}
    \sum_{n\leq x}\frac{1}{f(\log d(n))d(n)}
    &\quad\geq \sum_{n \in A(x)}\frac{1}{f(\log d(n))d(n)} \\
    &\quad\geq \frac{1}{f((\log 2)\log\log x)} \sum_{n\in A(x)}\frac{1}{d(n)} \\
        &\quad\geq \frac{1}{3}\frac{1}{f((\log 2)\log\log x)} \sum_{n\leq x}\frac{1}{d(n)}
  \end{align*}
  It is known (see for example~\cite[Theorem~II.6.8]{Tenenbaum}) that
  \begin{equation*}
    \sum_{n\leq x}\frac{1}{d(n)} \asymp \frac{x}{\sqrt{\log x}},
  \end{equation*}
  so we have 
  \begin{equation*}
    \sum_{n\leq x}\frac{1}{f(\log d(n))d(n)}
    \gg \frac{x}{f((\log 2)\log\log x) \sqrt{\log x}}
  \end{equation*}
  which finishes the proof.
\end{proof}

The following lemma consists of well-known facts about common arithmetic functions. 
\begin{lemma}[\cite{hardywright}]\th\label{lem:fromHW}
  For $n\in\NN$, let $\varphi(n)$ be the Euler totient, and
  \begin{equation*}
    \sigma_\tau(n):=\sum_{d\mid n} d^\tau.
  \end{equation*}
  Then the following asymptotic statements hold:
  \begin{align*}
    \sum_{n=1}^N\frac{\varphi(n)}{n} &\asymp \sum_{n=1}^N 1 \\
    \sum_{n=1}^N\sigma_1(n) &\asymp \sum_{n=1}^N n\\
    \sum_{n=1}^N\sigma_\tau(n) &\ll \sum_{n=1}^N 1
  \end{align*}
  for all $\tau < 0$. Also $\sigma_0(n) \ll n^{\eps}$ for every
  $\eps>0$.
\end{lemma}

Finally, we make frequent use of the following standard fact.

\begin{lemma}\th\label{lem:avgorder}
  Suppose $f:\NN\to \RR_{\geq 0}$ is a decreasing function and
  $g, h:\NN\to\NN$ are such that
  $\sum_{n=1}^N g(n) \asymp \sum_{n=1}^N h(n)$. Then
  $\sum_{n=1}^N f(n)g(n) \asymp \sum_{n=1}^N f(n)h(n)$.
\end{lemma}

\begin{proof}
By partial summation,
    \begin{align*}
      \sum_{n=1}^Nf(n)g(n)
      &=\sum_{n=1}^{N-1}\sum_{k=1}^n g(k)(f(n)-f(n+1))+f(N)\sum_{k=1}^N g(k)\\
      &\asymp \sum_{n=1}^{N-1}\sum_{k=1}^n h(k)(f(n)-f(n+1))+f(N)\sum_{k=1}^N h(k),
    \end{align*}
    since $f(n)-f(n+1)\geq 0$ for all $n$. Reversing the summation by
    parts gives $\sum_{n=1}^Nf(n)h(n)$, proving the lemma.
\end{proof}

\section{Fast divergence and extra divergence}

In this section we present the proof of~\th~\ref{thm:extradiv}.  But
first, we establish the following proposition, mentioned in the
introduction, showing that~\th~\ref{conj:movingKT} holds for all
$\psi$ such that $\sum_{q\leq Q}\psi(q)$ grows fast enough. It is a
direct consequence of an asymptotic formula appearing
in~\cite{Sprindzuk}.

\begin{proposition}\th\label{prop:sprindzuk}
  \th\ref{conj:movingKT} holds for $\psi:\NN\to\RR_{\geq 0}$ if there
  exists $\delta>0$ such that
  \begin{equation}
    \label{eq:sprindzukprofits}
    \sum_{q\leq Q} \psi(q) \gg (\log Q)(\log\log Q)^{3+\delta}
  \end{equation}
  holds for infinitely many $Q>1$.
\end{proposition}
  
\begin{proof}
  Denote $\Phi(Q) = \sum_{q\leq Q}2\psi(q)$. For $(\gamma_q)\in \RR$,
  $\psi:\NN\to\RR_{\geq 0}$ a decreasing function, $\alpha\in\RR$, and
  $Q\in\NN$, let
  \begin{equation*}
    N(Q,\alpha) = \set*{1 \leq q \leq Q : \norm{q\alpha - \gamma_q} < \psi(q)}.
  \end{equation*}
  Let $0 < \eps < \delta$. Sprind{\v z}uk~\cite[Page~51,
  Theorem~18]{Sprindzuk} gives the almost everywhere asymptotic
  formula
  \begin{equation}\label{eq:1}
    N(Q,\alpha) = \Phi(Q) + O\parens*{(\Phi(Q)\log Q)^{1/2}\brackets*{\log\parens*{\Phi(Q)\log Q}}^{3/2+\eps}}.
  \end{equation}
  Let $0 < \eps < \delta/2$. By~(\ref{eq:1}) there is some $C>0$ such
  that
  \begin{equation}\label{eq:3}
    N(Q,\alpha) \geq \Phi(Q) - C(\Phi(Q)\log Q)^{1/2}\brackets*{\log\parens*{\Phi(Q)\log Q}}^{3/2+\eps}
  \end{equation}
  for all large $Q>1$. Let
  \begin{equation*}
    M(Q) = \frac{\Phi(Q)}{\log Q},
  \end{equation*}
  so that $M(Q) >(\log\log Q)^{3+\delta}$ for infinitely many
  $Q$, by~(\ref{eq:sprindzukprofits}). Then~(\ref{eq:3}) becomes
  \begin{equation}\label{eq:4}
    N(Q,\alpha) \geq M(Q) \log(Q) - C M(Q)^{1/2}\log(Q)\brackets*{\log\parens*{M(Q)(\log(Q))^2}}^{3/2+\eps}.
  \end{equation}
  Note that
  \begin{align}
    C \brackets*{\log\parens*{M(Q)(\log(Q))^2}}^{3/2+\eps}
    &= C \brackets*{\log M(Q) + 2\log\log Q}^{3/2+\eps} \nonumber \\
    &\leq C \brackets*{\log M(Q) + 2\log\log Q}^{3/2+\eps}.\label{eq:5}
  \end{align}
  If, incidentally, $M(Q)\geq \log Q$ infinitely often, then for
  infinitely many $Q$, the bound becomes
  \begin{equation*}
    \leq C (3\log M(Q))^{3/2 + \eps} \leq M(Q)^{1/4}. 
  \end{equation*}
  Returning this to~(\ref{eq:4}) gives
  \begin{align*}
    N(Q,\alpha) &\geq M(Q) \log(Q) - M(Q)^{3/4}\log(Q)
  \end{align*}
  for all such $Q$, hence $N(Q,\alpha)\to\infty$.

  Suppose then that $M(Q) < \log Q$ for all large
  $Q$. Then~(\ref{eq:5}) becomes
  \begin{align*}
    C \brackets*{\log M(Q) + 2\log\log Q}^{3/2+\eps}
    &\leq C \brackets*{3\log\log Q}^{3/2+\eps} \\
    &\leq C' M(Q)^{1/2}(\log\log Q)^{\eps - \delta/2}
  \end{align*}
  for infinitely many $Q>1$, for some $C'>0$. Putting these into~(\ref{eq:4}) gives
  \begin{align*}
    N(Q,\alpha) &\geq M(Q)\log Q - C'M(Q)\log Q (\log\log Q)^{\eps - \delta/2}\\
                &\to \infty.
  \end{align*}
  Again, the result follows.
\end{proof}

As discussed in the introduction, the rate of growth required
in~\th\ref{prop:sprindzuk} is too fast to apply to even the most
natural functions, namely, those that are $O(1/q)$.

\subsection{Proof of~\th~\ref{thm:extradiv}}

We now turn our attention to proving~\th~\ref{thm:extradiv}. We will
need the following lemma. It is standard, but we include a brief
proof.

\begin{lemma}[Overlap estimates]\th\label{overlaps}
  Let $\gamma:=(\gamma_q)$ and $\psi:\NN\to [0,1/2]$. Define for
  each $q\in\NN$ the set
  \begin{equation*}
    A_q:= A_q(\psi,\gamma) = \set*{\alpha \in[0,1] : \norm{q\alpha - \gamma_q} < \psi(q)}.
  \end{equation*}
  Then for all $1 \leq r < q$, we have
  \begin{equation*}
    \meas(A_q\cap A_r)
    \leq 2 \meas(A_q)\meas(A_r)  + \frac{\gcd(q,r)}{q}\meas(A_q).
  \end{equation*}
\end{lemma}

\begin{proof}
  Let
  \begin{equation*}
    \delta = \min\set*{\frac{2\psi(q)}{q},\frac{2\psi(r)}{r}} \qquad\textrm{and}\qquad \Delta = \max\set*{\frac{2\psi(q)}{q},\frac{2\psi(r)}{r}}.
  \end{equation*}
  Note that $A_q\cap A_r$ is a union of intervals having length at
  most $\delta$, and that there are as many of these intervals as
  there are integer solutions $(a,b)$ to
  \begin{equation*}
    \abs*{\frac{a+\gamma_q}{q} - \frac{b + \gamma_r}{r}} < \frac{\psi(q)}{q}+\frac{\psi(r)}{r}\qquad 1\leq a\leq q\qquad 1\leq b\leq r.
  \end{equation*}
  In particular, we have
  \begin{equation}\label{eq:deltaN}
    \meas(A_q\cap A_r) \leq \delta N(q,r),
  \end{equation}
  where
  \begin{equation*}
    N(q,r)= \#\set{(a,b)\in \NN^2 : \abs{r(a+\gamma_q) - q(b+\gamma_r)} < \Delta qr,\, 1\leq a\leq q,\, 1\leq b\leq r}.
  \end{equation*}
  Since integers of the form $ar-bq$ differ by multiples of
  $\gcd(q,r)$, there are at most $(2\Delta qr/\gcd(q,r)) + 1$ such
  integers lying in the interval
  $((q\gamma_r - r\gamma_q)-\Delta qr, (q\gamma_r - r\gamma_q)+\Delta
  qr)$. And for each such integer, there are $\gcd(q,r)$ choices of
  $(a,b)$ that will achieve it. Therefore,
  \begin{equation}\label{eq:toberecalled}
    N(q,r) \leq \parens*{\frac{2\Delta qr}{\gcd(q,r)}+1}\gcd(q,r). 
  \end{equation}
  Putting this into~(\ref{eq:deltaN}) and observing that
  \begin{equation*}
    2\delta\Delta qr = 2\meas(A_q)\meas(A_r)
  \end{equation*}
  proves the lemma.  
\end{proof}

The following lemma is convenient for the application
of~\th\ref{BDVdensitylemma}.

\begin{lemma}\th\label{uniformity}
  Let $\gamma:=(\gamma_q)$ and $\psi:\NN\to [0,1/2]$. Define for each
  $q\in\NN$ the set $A_q$ as in~\th~\ref{overlaps}. Then for every
  nonempty open set $U\subset [0,1]$, there exists $q_0:=q_0(U)>0$
  such that
  \begin{equation*}
    \meas(A_q\cap U) \geq \frac{1}{2}\meas(A_q)\meas(U)
  \end{equation*}
  for all $q\geq q_0$.
\end{lemma}

\begin{proof}
  Suppose $I\subset[0,1]$ is a nonempty open interval. Then
  \begin{align}
    \meas(A_q\cap I)
    &\geq \brackets*{\#\parens*{\frac{1}{q}\ZZ\cap I}-2}\frac{2\psi(q)}{q} \nonumber\\
    &= \brackets*{q\meas(I)(1+o(1)) -2}\frac{2\psi(q)}{q} \nonumber\\
    &= q\meas(I)(1+o(1))\frac{2\psi(q)}{q} \nonumber\\
    &= \meas(A_q)\meas(I) (1+o(1)). \label{eq:AI}
  \end{align}
  Now, let $I_1, \dots, I_k\subset U$ be disjoint open intervals such that
  \begin{equation}\label{eq:23}
    \meas(I_1\cup\dots\cup I_k) = \sum_{j=1}^k\meas{I_j} \geq \frac{2}{3}\meas(U).
  \end{equation}
  Then
  \begin{align*}
    \meas(A_q\cap U)
    &\geq \sum_{j=1}^k \meas(A_q\cap I_j)\\
    &\overset{(\ref{eq:AI})}{=} (1+o(1))\sum_{j=1}^k \meas(A_q)\meas(I_j)\\
    &\overset{(\ref{eq:23})}{=} (1+o(1))\frac{2}{3}\meas(A_q)\meas(U) \\
    &\geq \frac{1}{2}\meas(A_q)\meas(U)
  \end{align*}
  for $q$ sufficiently large. 
\end{proof}

We are now prepared to state the proof of the first main theorem. 

\begin{proof}[Proof of~\th~\ref{thm:extradiv}]
  Let $\gamma:=(\gamma_q)$ and $\psi$ as in the theorem statement, and
  notice that no generality is lost if we assume that
  $\psi(q)\in[0,1/2]$ for all $q$. After all, if $\psi(q) > 1/2$
  infinitely often, then $A_q = [0,1]$ infinitely often and there is
  nothing to prove. This means we may assume $\psi(q) > 1/2$ holds for
  finitely many $q$, and in turn we may alter $\psi$ at all such $q$
  without sacrificing the assumptions of the theorem.

  For each $q\in\NN$, let $A_q$ be as in~\th~\ref{overlaps}. Then for
  all $1 \leq r < q$, \th~\ref{overlaps} gives
  \begin{equation*}
    \meas(A_q\cap A_r)
    \leq 2 \meas(A_q)\meas(A_r)  + \frac{\gcd(q,r)}{q}\meas(A_q).
  \end{equation*}
  Let $U\subset [0,1]$ be an open set. By~\th~\ref{uniformity},
  \begin{equation*}
    \meas(A_q\cap U) \geq \frac{1}{2}\meas(A_q)\meas(U)
  \end{equation*}
  for all $q$ sufficiently large. We are in a position to
  apply~\th~\ref{prop:abstractyu} with $\kappa = 2$ and
  $\eta(q,r) = \frac{\gcd(q,r)}{q}$. Note that
  \begin{equation*}
    \eta(q) = \sum_{1\leq r \leq q}\frac{\gcd(q,r)}{q} = \sum_{d\mid q}\frac{\varphi(d)}{d} \leq d(q), 
  \end{equation*}
  where $d(q)$ is the number of divisors of $q$. Put
  \begin{equation*}
    f(x) = x (\log x) (\log\log x) \dots (\underbrace{\log\log\dots\log}_{k-2 \textrm{ iterates}} x)^{1+\eps}.
  \end{equation*}
  In this case,
  \begin{equation}\label{eq:log2}
    \sum_{n\leq x}\frac{1}{f((\log 2)\log\log n)\sqrt{\log n} }\asymp \sum_{n\leq x}\frac{1}{f(\log\log n)\sqrt{\log n} } \asymp \frac{x}{f(\log\log x)\sqrt{\log x}}.
  \end{equation}
  Since, by assumption, 
  \begin{equation*}
    \sum_{q=1}^\infty \frac{\psi(q)}{f(\log\log q)\sqrt{\log q}} = \infty,
  \end{equation*}
  and $\psi$ is decreasing, it follows from~\eqref{eq:log2} and~\th\ref{loglogsum,lem:avgorder} that
  \begin{equation*}
    \sum_{q=1}^\infty\frac{\psi(q)}{f(\log d(q))d(q)} = \infty.
  \end{equation*}
  Put
  \begin{equation*}
      h(x) = f(\log x)
  \end{equation*}
  and note that~\eqref{eq:hconv} and~(\ref{eq:etasum}) are
  satisfied. Now~\th\ref{prop:abstractyu} implies
  \begin{equation*}
      \meas(W(\psi,\gamma)\cap U) \geq \frac{\meas(U)^2}{8}.
  \end{equation*}
  Since the open set $U\subset [0,1]$ was arbitrary,
  \th\ref{BDVdensitylemma} gives $\meas(W(\psi,\gamma)) =1$ and the
  theorem is proved.
\end{proof}

\section{QIA for a fixed target center}

Most proofs of Khintchine-type results work by first establishing
quasi-independence on average for the sets that arise in the problem.
The idea for the proof of~\th~\ref{thm:finitetarget} is a pigeonholing
argument that takes as input the quasi-independence that comes from
proofs of the inhomogeneous Khintchine theorem (with a fixed target
center). One such proof is to be found in Schmidt~\cite{Schmidt}. The
purpose of this section is to prove~\th\ref{prop:qia}, a
quasi-independence result that can be deduced
from~\cite[Proposition~2]{Schmidt}. We include its proof for
completeness, and also so that the reader can see (in~\th~\ref{lem:N})
where in its proof it is important that the inhomogeneous parameter is
fixed.

Let $\gamma\in\RR$. For each $q\in\NN$, let $A:=A(q), B:=B(q)$ be positive
integers satisfying $(A,B)=1$ and
\begin{equation}\label{eq:AandB}
  1\leq B \leq q^{1/2}\quad\textrm{and}\quad \abs*{\gamma - \frac{A}{B}} < \frac{1}{q^{1/2}B},
\end{equation}
and define
\begin{equation}\label{eq:S}
  S(q)= \set{a=0, \dots, q-1 : (aB + A, q)=1}.
\end{equation}
Notice that if $\gamma\in \QQ$, then for all $q$ large enough, we have
$\gamma = A(q)/B(q) = A/B$.

Let
\begin{align}\label{eq:partialrestrictedA}
  A_q'(\psi,\gamma)
  &= \set*{\alpha \in [0,1] : \abs{q\alpha - p - \gamma}<\psi(q) \textrm{ for some } p\in S(q)}\\
  &= \bigcup_{a\in S(q)}\parens*{\frac{a+\gamma - \psi(q)}{q}, \frac{a+\gamma + \psi(q)}{q}}\cap [0,1]. \nonumber
\end{align}
Then $W(\psi, \gamma) = \limsup_{q\to\infty} A_q'(\psi, \gamma)$, and
we aim to establish quasi-independence no average for the sets
$A_q'$. The following lemma is key in the calculation.

\begin{lemma}[{\cite[Lemma~13]{Schmidt}}]\th\label{lem:N}
  Let $\gamma\in\RR$, $q\in\NN$, and
  $0 < \delta \leq 1/4$. Let $A:=A(q), B:=B(q), S:=S(q)$ be as
  in~(\ref{eq:AandB}) and~(\ref{eq:S}). For $r\in\NN$, put
  \begin{equation*}
    N_\delta(q,r) = \#\set{(a,b)\in S\times \set{0, \dots, r-1} : \underbrace{\abs{(a+\gamma)r-(b+\gamma)q} < (q,r)^{\delta}}_\star}. 
  \end{equation*}
  Then
  \begin{equation*}
    \sum_{\substack{r < q \\ (q,r) = d}} N_\delta(q,r) \ll qd^{-1/4} + d
  \end{equation*}
  for every $d\mid q$.
\end{lemma}

\begin{proof}
  Let $d\mid q$. We proceed in cases.

  \subsubsection*{Case 1}

  Assume $2B < d^{1/4}$ and $2Bq'\abs{\gamma - A/B} < 1$.  Then for
  $r<q$ with $(q,r)=d$, the inequality~$\star$ implies that
  \begin{equation*}
    \abs*{(a + A/B)r - (b+A/B)q + (r-q)(\gamma - A/B)} < d^{1/4}
  \end{equation*}
  hence
  \begin{equation*}
    \abs*{\parens*{a + \frac{A}{B}}r - \parens*{b+\frac{A}{B}}q}< d^{1/4}  + (q-r)\parens*{\gamma - \frac{A}{B}} < \frac{d}{B}.
  \end{equation*}
  Since $B\geq 1$, the inequality implies
  \begin{equation*}
    \abs*{\parens*{a + \frac{A}{B}}r - \parens*{b+\frac{A}{B}}q}=0,
  \end{equation*}
  which has no solutions $r < q$ because $(aB + A, q)=1$. We remark
  that this is the only place where this comprimality condition is
  used.

  \subsubsection*{Case 2}

  Assume $2B\geq d^{1/4}$ or $2Bq'\abs{\gamma - A/B} \geq 1$. For
  $r<q$ such that $(q,r)=d$ the inequality~$\star$ implies 
  \begin{equation}\label{eq:case2}
    \abs{(a+\gamma)r'-(b+\gamma)q'} < d^{-3/4}
  \end{equation}
  where $q' = q/d$ and $r' = r/d$. This in turn implies that
  \begin{equation}\label{eq:case2implication}
    \norm{(q'-r')\gamma} < d^{-3/4}. 
  \end{equation}
  We aim to bound the number of $r' = 0, \dots, q'-1$ such that the
  above inequality holds or, equivalently,
  \begin{equation*}
    L(q,r):=\#\set{k=1, \dots, q' : \norm{k\gamma} < d^{-3/4}}.
  \end{equation*}
  For a given $1\leq r < q$, there are at most $2$ integers $\ell$
  such that
  \begin{equation*}
    \abs{(q'-r')\gamma - \ell} < d^{-3/4},
  \end{equation*}
  and for each such $q,r,\ell$, at most $d$ choices of $(a,b)$ that
  will satisfy~(\ref{eq:case2}), hence, we have
  \begin{equation}\label{eq:multbyd}
    \sum_{\substack{r < q \\ d \mid r}}N_\delta(q,r) \leq  2d L(q,r).
  \end{equation}
  There are two sub-cases.

  The first subcase is $2B\geq d^{1/4}$.  Note that
  \begin{equation}\label{eq:subcase1}
    \norm{k\gamma} = \norm*{k\parens*{\frac{A}{B} + (\gamma - A/B)}} = \abs*{k\parens*{\frac{A}{B} + (\gamma - A/B)} - \ell}
  \end{equation}
  for some $\ell\in\ZZ$. In particular, $\norm{k\gamma}<d^{-3/4}$ implies 
  \begin{equation*}
    \norm*{\frac{kA}{B} - \ell} < d^{-3/4} + k\abs*{\gamma - \frac{A}{B}}
  \end{equation*}
  Since $(A,B)=1$, the fractional part $\set{kA/B + \ell}$ cycles
  through $B$ distinct points in $[0,1]$ as
  $k=1, \dots, B, \dots, q'$, separated by $1/B$. So we may
  bound the number of solutions (with $k=1, \dots, B$) to
  \begin{equation*}
    \norm*{\frac{kA}{B} - \ell} < d^{-3/4} + B\abs*{\gamma - \frac{A}{B}} \leq d^{-3/4}  + q^{-1/2}
  \end{equation*}
  by $B(d^{-3/4} + q^{-1/2})+1\ll Bd^{-3/4}+1$. As $k=1, \dots, q'$,
  the fractional party $\set{kA/B + \ell}$ completes $\leq qd^{-1}+1$
  cycles, so we may bound
  \begin{align*}
    L(q,r)
    &\ll \parens*{Bd^{-3/4} +1}\parens*{\frac{q'}{B} + 1}\\
    &\ll qd^{-7/4} + Bd^{-3/4}+\frac{qd^{-1}}{B} + 1 \\
    &\leq qd^{-7/4} + q^{1/2}d^{-3/4}+ 2qd^{-5/4} + 1 \\
    &\ll qd^{-5/4} + 1.
  \end{align*}
  Then, by~(\ref{eq:multbyd}),
  \begin{equation*}
    \sum_{\substack{r < q \\ d \mid r}}N_\delta(q,r) \ll qd^{-1/4} + d,
  \end{equation*}
  which settles this subcase.

  The only remaining subcase is $2B < d^{1/4}$ and
  $2Bq'\abs{\gamma - A/B} \geq 1$. Then~(\ref{eq:subcase1}) leads to
  \begin{equation}\label{eq:subcase2}
    \abs*{\frac{m}{B} + k\parens*{\gamma - \frac{A}{B}}} < d^{-3/4}
  \end{equation}
  for some integer $m$ satisfying
  \begin{equation*}
    \abs{m} < \parens*{d^{-3/4} + q' \abs*{\gamma - \frac{A}{B}}}B.
  \end{equation*}
  For any such $m$ (fixed), the number of solutions $k=1, \dots, q'$
  to~(\ref{eq:subcase2}) is bounded by
  \begin{equation*}
    2d^{-3/4}\abs*{\gamma - \frac{A}{B}}^{-1} + 1. 
  \end{equation*}
  Multiplying, we find that
  \begin{align*}
    L(q,r)
    &\ll \parens*{\parens*{d^{-3/4} + q' \abs*{\gamma - \frac{A}{B}}}B + 1}\parens*{2d^{-3/4}\abs*{\gamma - \frac{A}{B}}^{-1} + 1}\\
    &\ll d^{-3/2}B^2q' + Bd^{-3/4} + q'B d^{-3/4} + q'q^{-1/2} + 1\\
    &\ll q d^{-5/4} + 1.
  \end{align*}
  Now, by~(\ref{eq:multbyd}), 
  \begin{equation*}
    \sum_{\substack{r < q \\ d \mid r}}N_\delta(q,r) \ll q d^{-1/4} + d,
  \end{equation*}
  which settles the second case.
\end{proof}

\begin{proposition}[{\cite[Proposition~2]{Schmidt}}]\th\label{prop:qia}
  Let $\gamma\in \RR$ and $\psi:\NN\to\RR_{\geq 0}$
  decreasing and satisfying $\psi(q) \leq 1/q$ and $\sum\psi(q)=\infty$. Then
  \begin{equation*}
    \sum_{q,r=1}^Q \meas(A_q'(\psi,\gamma)\cap A_r'(\psi,\gamma)) \ll \Psi(Q)^2
  \end{equation*}
  where $\Psi(Q) = \sum_{q=1}^Q\psi(q)$.
\end{proposition}

\begin{proof}
  For brevity, denote $A_q':=A_q'(\psi,\gamma)$, where
  $A_q'(\psi,\gamma)$ is as in~(\ref{eq:partialrestrictedA}). For
  $q,r\in\NN$, define
  \begin{equation*}
    M(q,r) = \#\set{(a,b)\in S(q)\times S(r) : \abs{(a+\gamma)r-(b+\gamma)q} < 2r\psi(q)},
  \end{equation*}
  Then for $r < q$, 
  \begin{equation*}
    \meas(A_q'\cap A_r') \leq \frac{2\psi(q)}{q} M(q,r),
  \end{equation*}
  trivially. It follows that
  \begin{align}
    \sum_{q,r=1}^Q \meas(A_q'\cap A_r')
    &\ll \sum_{q=1}^Q\frac{\psi(q)}{q}\sum_{r=1}^qM(q,r) \nonumber \\
    &\ll \sum_{q=1}^Q\frac{\psi(q)}{q}\sum_{\substack{r=1 \\ q\psi(r) \geq (q,r)}}^q M (q,r)  \label{eq:first} \\
    &\quad + \sum_{q=1}^Q\frac{\psi(q)}{q}\sum_{\substack{r=1 \\ (q,r)> q\psi(r) \geq (q,r)^{1/5}}}^qM(q,r)  \label{eq:second} \\
    &\quad + \sum_{q=1}^Q\frac{\psi(q)}{q}\sum_{\substack{r=1 \\ q\psi(r) < (q,r)^{1/5}}}^qM(q,r) \label{eq:third}
  \end{align}
  For $q\psi(r) \geq (q,r)$ we have by~(\ref{eq:toberecalled})
  \begin{equation*}
    M(q,r) \ll \parens*{\frac{q\psi(r)}{(q,r)} + 1} (q,r) \ll q\psi(r)
  \end{equation*}
  so the first term~\eqref{eq:first} is
  \begin{align*}
    \sum_{q=1}^Q\frac{\psi(q)}{q}\sum_{\substack{r=1 \\ q\psi(r) \geq (q,r)}}^q M(q,r)
    &\ll \sum_{q=1}^Q\frac{\psi(q)}{q}\sum_{\substack{r=1 \\ q\psi(r) \geq (q,r)}}^qq\psi(r)\\
    &\leq \sum_{q=1}^Q\psi(q) \sum_{r=1 }^q \psi(r)\\
    &\ll \Psi(Q)^2. 
  \end{align*}
  When $(q,r) \geq q\psi(r)$ one has $M(q,r) \ll(q,r)$
  by~\ref{overlaps}. The second term~\eqref{eq:second} becomes
  \begin{align*}
    \sum_{q=1}^Q\frac{\psi(q)}{q}\sum_{\substack{r=1 \\ (q,r) \geq q\psi(r) \geq (q,r)^{1/5}}}^qM (q,r)
    &\leq \sum_{q=1}^Q\frac{\psi(q)}{q} \sum_{\substack{r=1 \\ q/r\geq (q,r)^{1/5}}}^q(q,r)\\
    &\leq \sum_{q=1}^Q\frac{\psi(q)}{q} \sum_{d\mid q}d \sum_{\substack{(q,r)=d \\ r \leq q d^{-1/5}}} 1 \\
    &\leq \sum_{q=1}^Q\frac{\psi(q)}{q} \sum_{d\mid q} q d^{-1/5} \\
    &\leq \sum_{q=1}^Q\psi(q)\sigma_{-1/5}(q) \\
    &\ll \Psi(Q),
  \end{align*}
  where we have used~\th~\ref{lem:fromHW,lem:avgorder}.

  The third term~\eqref{eq:third} is
  \begin{align*}
    \sum_{q=1}^Q\frac{\psi(q)}{q}\sum_{\substack{r=1 \\ q\psi(r) < (q,r)^{1/5}}}^q M(q,r)
    &\leq \sum_{q=1}^Q\frac{\psi(q)}{q}\sum_{\substack{r=1 \\ q\psi(r) < (q,r)^{1/5}}}^q N_{1/5}(q,r),
  \end{align*}
  where $N_{1/5}(q,r)$ is as in~\th~\ref{lem:N}. By that same lemma,
  \begin{align*}
    \cdots &\leq \sum_{q=1}^Q\frac{\psi(q)}{q} \sum_{d\mid q} \sum_{\substack{r<q \\ (q,r)=d}} N_{1/5}(q,r)\\
           &\leq \sum_{q=1}^Q\frac{\psi(q)}{q} \sum_{d\mid q} (qd^{-1/4} + d) \\
           &\leq \sum_{q=1}^Q\psi(q)\sigma_{-1/4}(q) + \sum_{q=1}^Q\psi(q)\sigma_{-1}(q) \\
    &\ll \Psi(Q),
  \end{align*}
  By~\th~\ref{lem:fromHW,lem:avgorder}. Combining, the result is
  proved.
\end{proof}

\section{Finitely centered target}

In this section we prove~\th~\ref{thm:finitetarget}. We start with the
following equidistribution statement. It is useful for the application
of~\th~\ref{BHV}.

\begin{lemma}\th\label{equidistribution}
  Let $\gamma\in\RR$ and for each $q\in\NN$, let $S(q)$ be as
  in~(\ref{eq:S}). Then for every $q$ we have
  $\#S(q) \geq \varphi(q)$. Furthermore, the sets
  $\set*{\frac{1}{q} S(q)}_{q\geq 1}$ equidistribute in $[0,1]$.
\end{lemma}

\begin{remark*}
  The assertions of this lemma are known. Schmidt~\cite{Schmidt}
  proves $\#S(q)\geq \varphi(q)$, and the equidistribution statement
  can be found in the proof
  of~\cite[Proposition~1]{beresnevich2024borel} using a different
  argument from the one below. We include proof here for completeness.
\end{remark*}

\begin{proof}
  The condition $\gcd(aB-A, q)=1$ is equivalent to $\gcd(aB-A, q')=1$
  where $q = q' q''$ and every prime divisor of $q''$ also divides $B$
  and $\gcd(B,q')=1$. Now
\begin{align*}
    \#S(q) &= \#\set{a=1, \dots, q : \gcd(aB-A,q')=1} \\
    &= \varphi(q')(q/q') \\
    &= \varphi(q')q'' \\
    &\geq \varphi(q),
\end{align*}
proving the first assertion. 

Now let $I\subset [0,1]$ be a nonempty interval. We have
\begin{align*}
  \#\parens*{\frac{S(q)}{q}\cap I} 
  &= \sum_{a=1}^q \bone_I (a/q)\bone_{\set{1}}(\gcd(aB-A,q)) \\
  &= \sum_{a=1}^q \bone_I(a/q)\sum_{d\mid \gcd(aB-A,q)} \mu(d) \\
  &= \sum_{d\mid q}\mu(d)\sum_{d \mid aB-A}\bone_{I}(a/q).
\end{align*}
Observe that $\set{a\in \ZZ : d \mid aB-A}$ is either an arithmetic
sequence of integers or empty, and in either case we have
\begin{equation*}
  \sum_{d \mid aB-A}\bone_{I}(a/q) = \meas(I) \sum_{\substack{a=1 \\ d \mid aB-A}}^q 1 + O(1). 
\end{equation*}
Hence,
\begin{align*}
  \#\parens*{\frac{S(q)}{q}\cap I} 
  &= \meas(I)\sum_{d\mid q}\mu(d)\sum_{d \mid aB-A}\bone_{I}(a/q) + O\parens*{\sum_{d\mid q} \abs{\mu(d)}} \\
  &= \meas(I) \# S(q)(1 + o(1)),
\end{align*}
after recalling that $\# S(q) \geq \varphi(q)$.
\end{proof}

\th~\ref{prop:qia} contains the assumption that $\psi(q) \leq 1/q$ for
all $q$. The following lemma shows that no generality is lost with
this assumption.

\begin{lemma}[Reduction to $O(1/q)$]\th\label{lem:wlog}
  Suppose $\psi:\NN\to\RR_{\geq 0}$ is decreasing and
  $\sum_{q=1}^\infty \psi(q)=\infty$. Then
  \begin{equation*}
    \sum_{q=1}^\infty \min\set*{\psi(q), \frac{1}{q}}=\infty. 
  \end{equation*}
\end{lemma}

\begin{proof}
  Let $\tilde\psi(q) = \min\set{\psi(q), 1/q}$ and note that
  $\tilde\psi$ is decreasing. By Cauchy's condensation test, its
  divergence is equivalent to that of
  $\sum_{k\geq 0} 2^k\tilde\psi(2^k)$. But
\begin{align*}
    \sum_{k\geq 0}2^k \tilde\psi(2^k) 
    &= \sum_{k\geq 0}\min\set{2^k\psi(2^k), 1} \\
    &= \sum_{\substack{k\geq 0 \\ \tilde\psi(2^k) = \psi(2^k)}}2^k\psi(2^k) + \sum_{\substack{k\geq 0 \\ \tilde\psi(2^k) < \psi(2^k)}} 1.
\end{align*}
If there are infinitely many $k\geq 0$ such that
$\tilde\psi(2^k) < \psi(2^k)$, then the second sum diverges, and we
are done. Otherwise, the first sum has the same tail as
$\sum_{k\geq 0}2^k\psi(2^k)$, which diverges, because $\sum \psi(q)$
diverges. Once again, we are done.
\end{proof}

\begin{proof}[Proof of~\th\ref{thm:finitetarget}]
  Let $\set{\pi_1, \dots, \pi_\ell}$,
  $\set{\sigma_1, \dots, \sigma_\ell}$, and $\psi$ be as in the
  theorem statement. By~\th\ref{lem:wlog}, we may assume that
  $\psi(q) \leq 1/q$ for each $q$. For each $q\in \NN$ and
  $k\in\set{1,\dots, \ell}$, denote $A_{q,k}':= A_q'(\psi,\sigma_k)$
  where the latter is as in~\eqref{eq:partialrestrictedA}. Notice that
  \begin{equation*}
    W(\bone_{\pi_k}\psi, \sigma_k) =  \limsup_{\substack{q\to\infty \\ q\in\pi_k}} A_{q,k}'
  \end{equation*}
  for each $k$. 

  By the pigeonhole principle, there is some $k\in\set{1,\dots, \ell}$ (which we now fix)
  such that
  \begin{equation*}
    \sum_{q \in \pi_{k} \cap [1,Q] } \frac{\varphi(q)\psi(q)}{q} \geq \frac{1}{\ell} \sum_{q \in [1, Q]} \frac{\varphi(q)\psi(q)}{q}
  \end{equation*}
  for infinitely many $Q$. In particular, for this $k$ we have
  \begin{equation}\label{eq:cond1}
    \sum_{q \in \pi_{k} \cap [1,Q] } \meas(A_{q,k}') \geq \frac{1}{\ell} \sum_{q \in [1, Q]} \frac{\varphi(q)\psi(q)}{q},
  \end{equation}
  by~\th~\ref{equidistribution}. We now show that
  conditions~(\ref{eqn01}),~(\ref{eqn02}), and~(\ref{vb89}) are
  satisfied by the sets $A_{q,k}'\, (q\in\pi_k)$. 

  First, since $\psi$ is decreasing and $\sum \psi(q)$ diverges, it
  follows that $\sum \frac{\varphi(q)\psi(q)}{q}$ also diverges,
  by~\th~\ref{lem:fromHW,lem:avgorder}. Condition~(\ref{eqn01}) is now
  a consequence of~(\ref{eq:cond1}).

  Next, we have
  \begin{align*}
    \sum_{q,r \in \pi_k \cap [1,Q] }\meas(A_{q,k}' \cap A_{r,k}') 
    &\leq \sum_{q,r \in [1,Q]} \meas(A_{q,k}' \cap A_{r,k}') \\
    &\ll \Psi(Q)^2
  \end{align*}
  by~\th~\ref{prop:qia}. Again
  by~\th~\ref{lem:fromHW,lem:avgorder} and the monotonicity of $\psi$,
  this implies
  \begin{align*}
    \sum_{q,r \in \pi_k \cap [1,Q] }\meas(A_{q,k}' \cap A_{r,k}') 
    &\ll \parens*{\sum_{q\leq Q}\frac{\varphi(q)\psi(q)}{q}}^2 \\
    & \ll \ell^2\parens*{\sum_{q \in \pi_k \cap [1,Q]} \meas(A_{q,k}')}^2,
  \end{align*}
  for infinitely many $Q$, by~(\ref{eq:cond1}). This
  establishes~(\ref{eqn02}). 

  Finally, let $I \subset [0,1]$ be an interval. We have
  \begin{align*}
    \meas(A_{q,k}'\cap I)
    &\leq \frac{\psi(q)}{q}\#\parens*{\frac{S(q)}{q}\cap I} \\
    &= \frac{\psi(q)}{q}\parens*{\# S(q)}\meas(I)(1+o(1)) \\
    &= \meas(A_{q,k}')\meas(I)(1 + o(1))
  \end{align*}
  by~\th~\ref{equidistribution} and the fact that $\psi(q) \leq
  1/q$. This implies that~(\ref{vb89}) is satisfied. Now,
  by~\th~\ref{BHV},
  \begin{equation*}
    \meas\parens*{\limsup_{\substack{q\to\infty \\ q\in\pi_k}} A_{q,k}'}=1.
  \end{equation*}
  Therefore $\meas(W(\bone_{\pi_k}\psi, \sigma_k))=1$ and the proof is finished.
\end{proof}

\subsection*{Acknowledgments}

We thank Manuel Hauke for the helpful comments he gave us on a
previous draft of this article.

\bibliographystyle{plain}
\bibliography{bibliography.bib}




\end{document}